\documentclass[11pt]{amsart}

\usepackage{amsmath, amsthm}
\usepackage{amssymb}
\usepackage{amsfonts}
\usepackage{latexsym}
\usepackage{mathrsfs}
\usepackage{bm}
\usepackage{graphicx}

\newcommand{\grad}{\nabla}

\renewcommand{\L}{\mathcal{L}}

\DeclareMathOperator{\re}{Re}
\renewcommand{\Re}{\re}

\renewcommand{\tilde}[1]{\widetilde{#1}}

\renewcommand{\leq}{\leqslant}
\renewcommand{\geq}{\geqslant}

\DeclareMathOperator{\Vol}{Vol}
\DeclareMathOperator{\dVol}{dVol}

\DeclareMathOperator{\Div}{div}

\DeclareMathOperator{\Sca}{Sca}
\DeclareMathOperator{\Diff}{Diff}

\newcommand{\Mod}{\mbox{\rm Mod}}

\newcommand{\param}{{\mathchoice{\mkern1mu\mbox{\raise2.2pt\hbox{$\centerdot$}}\mkern1mu}{\mkern1mu\mbox{\raise2.2pt\hbox{$\centerdot$}}\mkern1mu}{\mkern1.5mu\centerdot\mkern1.5mu}{\mkern1.5mu\centerdot\mkern1.5mu}}}
\numberwithin{equation}{section}

\theoremstyle{plain}
\newtheorem{theorem}{Theorem}[section]

\newtheorem{lemma}[theorem]{Lemma}

\newtheorem{conjecture}[theorem]{Conjecture}

\theoremstyle{definition}

\theoremstyle{definition}
\newtheorem*{remarksenv}{Remarks}

\begin{document}

\date{}
\title[Scalar Curvature]
{Scalar curvatures of Hermitian metrics on the moduli space of Riemann surfaces}

\author{Yunhui Wu}

\address{Department of Mathematics\\
         Rice University\\
         Houston, Texas, 77005-1892\\}
\email{yw22@rice.edu}

\begin{abstract} 
In this article we show that any finite cover of the moduli space of closed Riemann surfaces of $g$ genus with $g\geq 2$ does not admit any complete finite-volume Hermitian metric of non-negative scalar curvature.
%, which confirms a conjecture of Farb-Weinberger in the Hermitian case. 

Moreover, we also show that the total mass of the scalar curvature of any almost Hermitian metric, which is equivalent to the Teichm\"uller metric, on any finite cover of the moduli space is negative provided that the scalar curvature is bounded from below.
\end{abstract}

\subjclass{30F60, 32G15}
\keywords{Moduli space, Scalar curvature, Teichm\"uller metric}

\maketitle

\section{Introduction}
Let $S_g$ be a closed surface of genus $g$ with $g\geq 2$, $\Mod(S_g)$ be the mapping class group and $T_{g}$ be the Teichm\"uller space of $S_g$. Topologically $T_{g}$  is a complex manifold of complex dimension $3g-3$, which carries various $\Mod(S_g)$-invariant metrics which descend into metrics on the moduli space $\mathbb{M}_g$ of $S_g$ with respective properties. For examples, the Teichm\"uller metric, Kobayashi metric and Carathe\'odory metric are complete and Finsler. The Weil-Petersson metric is K\"ahler \cite{Ahlfors}, incomplete \cite{Chu,  Wolpert1} and has negative sectional curvatures \cite{Wol86}. The Asymptotic Poincar\'e metric, Induced Bergman metric, K\"ahler-Einstein metric, McMullen metric, Ricci metric, and perturbed Ricci metric are complete and K\"ahler. In \cite{LSY04, LSY05, McM00, Yeu04}, the authors showed that the metrics listed above except the Weil-Petersson metric are equivalent. 

It is shown that the perturbed Ricci metric \cite{LSY04, LSY05} has pinched negative Ricci curvature. So does the scalar curvature of the perturbed Ricci metric. And the McMullen metric \cite{McM00} has negative scalar curvature at certain points since the metric, restricted on certain thick part of the moduli space, is the Weil-Petersson metric. However, in \cite{FW-Scalar} Farb and Weinberger show that any finite cover $M$ of the moduli space $\mathbb{M}_g$ $(g\geq 2)$ admits a complete finite-volume Riemannian metric of (uniformly bounded) positive scalar curvature. They also show that this metric is not quasi-isometric to the Teichm\"uller metric. And they conjecture (see Conjecture 4.6 in \cite{Farb-pro})  
%The K\"ahler-Einstein metric \cite{ChengYau80} has negative constant Ricci curvature. In particular the scalar curvature of the K\"ahler-Einstein metric is also a negative constant. 
\begin{conjecture}[Farb-Weinberger]\label{fwc}
Let $S_{g}$ be a surface of $g$ genus with $g\geq 2$. Then any finite cover $M$ of the moduli space $\mathbb{M}_{g}$ of $S_{g}$ does not admit a finite volume Riemannian metric of (uniformly bounded) positive scalar curvature in the quasi-isometry class of the Teichm\"uller metric.
\end{conjecture}

Let $M$ be any finite cover of $\mathbb{M}_g$. The natural complex structure on the Teichm\"uller space descends into a complex structure on $M$. In this paper we will focus complete Hermitian metrics on $M$ with this complex structure. Our first result is  

\begin{theorem}\label{mt-1}
Let $S_{g}$ be a surface of $g$ genus with $g\geq 2$ and $M$ be a finite cover of the moduli space $\mathbb{M}_{g}$ of $S_{g}$. Then for any complete finite-volume Hermitian metric $||\cdot||$ on $M$, the scalar curvature $\Sca$ of $(M,||\cdot||)$ satisfies 
$$\inf_{p \in M} \Sca(p)<0.$$
\end{theorem}

As introduced before there is a list of canonical metrics which are equivalent (or quasi-isometric) to the Teichm\"uller metric $||\cdot||_T$ (see \cite{LSY04, LSY05, McM00, Yeu04}). Our second aim is the following uniform upper bound on the infimum of the scalar curvature of any Hermitian metric on a given class. 

\begin{theorem}\label{mt-2}
Let $S_{g}$ be a surface of $g$ genus with $g\geq 2$ and $M$ be a finite cover of the moduli space $\mathbb{M}_{g}$ of $S_{g}$. Given two constants $k_1, k_2>0$, then for any Hermitian metric $||\cdot||$ on $M$ with $  k_1 ||\cdot|| \leq ||\cdot||_T \leq k_2 ||\cdot||$, there exists a constant $K(k_1, g)>0$ only depending on $k_1$ and $g$ such that
the scalar curvature satisfies
$$\inf_{p \in (M,||\cdot||)} \Sca(p)\leq -K(k_1,g)<0.$$
\end{theorem}

Theorem \ref{mt-2} confirms Conjecture \ref{fwc} for the Hermitian case. For the general Riemannian case, we have been told by Prof. Farb in \cite{Farb-p-c} that Conjecture \ref{fwc} is confirmed by him and Weinberger in a forthcoming paper. And their method is completely different from the method in this article. \\

Both Theorem \ref{mt-1} and Theorem \ref{mt-2} tell that the scalar curvature of any Hermitian metric, which is equivalent to the Teichm\"uller metric, on the moduli space is negative at certain points. It is natural to ask \textsl{whether the total mass of the scalar curvature could be positive}. Our last result tells that this is impossible if we assume the metric is almost Hermitian and its scalar curvature has a low bound.

\begin{theorem}\label{mt-3}
Let $S_{g}$ be a surface of $g$ genus with $g\geq 2$ and $M$ be a finite cover of the moduli space $\mathbb{M}_{g}$ of $S_{g}$. Let 
$||\cdot|$ be an almost Hermitian metric on $M$ satisfying $||\cdot|| \approx ||\cdot||_T $. Assume that the scalar curvature of $(M,||\cdot||)$ is bounded from below, then the total mass of the scalar curvature satisfies
$$\int_{p \in (M,||\cdot||)} \Sca(p)\dVol(p)<0.$$
\end{theorem}

Theorem \ref{mt-3} applies to the Asymptotic Poincar\'e metric, Induced Bergman metric, McMullen metric and Ricci metric, which is new. Actually it also applies to any metric in the the convex hull of the K\"ahler-Einstein metric, perturbed Ricci metric and the four metrics above.\\

It is \textsl{interesting} to know \textsl{whether there exists a Hermitian metric $||\cdot||$ on a finite cover $M$ of  $\mathbb{M}_g$ such that $||\cdot||$ is equivalent to the Teichm\"uller metric and $(M,||\cdot||)$ has non-negative scalar curvature outside some compact subset of $M$.} We hope the method in this paper is helpful for this question.
\quad 

\subsection{Plan of the paper.} In section 2 we review some recent developments on the canonical metrics on the moduli space of surfaces and recall one formula of S. S. Chern which is crucial for this article. In section \ref{sec-3} we establish Theorem \ref{mt-1}. Theorem \ref{mt-2} is proved in section \ref{sec-4}. And we will finish the proof of Theorem \ref{mt-3} in section \ref{sec-5}. 

\section{Notations and Preliminaries}\label{np} 

\subsection{Surfaces}
Let $S_g$ be a closed surface of $g$ genus with $g\geq 2$, and $\textsl{M}_{-1}$ denote the space of Riemannian metrics with constant curvature $-1$, and $X=(S_g,\sigma|dz|^2)$ be an element in $\textsl{M}_{-1}$. The group $\Diff_0$ of diffeomorphisms of $S_g$ isotopic to the identity, acts by pull-back on $\textsl{M}_{-1}$. The Teichm\"uller space $T_{g}$ of $S_g$  is defined by the quotient space
\begin{equation}
\nonumber T_g=M_{-1}/\Diff_0.  
\end{equation}

Let $\Diff_+$ be the group of orientation-preserving diffeomorphisms of $S_g$. The mapping class group $\Mod(S_g)$ is defined as
\begin{equation}
\nonumber \Mod(S_g)=\Diff_+/\Diff_0.  
\end{equation}

The moduli space $\mathbb{M}_g$ of $S_g$ is defined by the quotient space
\begin{equation}
\nonumber \mathbb{M}_g=T_g/\Mod(S_g).  
\end{equation}

The Teichm\"uller space has a natural complex structure, and its holomorphic cotangent space $T_{X}^{*}T_{g}$ is identified with the \textsl{quadratic differentials} 
$$QD(X)=\{\phi(z)dz^2\}$$ 
while its holomorphic tangent space is identified with the \textsl{harmonic Beltrami differentials} 
$$HBD(X)=\{\frac{\overline{\phi(z)}}{\sigma(z)}\frac{d\overline{z}}{dz}\}.$$ 

Recall that the \textsl{Teichm\"uller metric} $||\cdot||_T$ on $T_g$ is defined as
$$||\frac{\overline{\phi(z)}}{\sigma(z)}||_T:=\sup_{\psi dz^2 \in QD(X), \ \ \int_{X}|\psi|=1} \Re  \int_X {\frac{\overline{\phi(z)}}{\sigma(z)} \cdot \psi(z)} \frac{dz\wedge d\overline{z}}{-2\textbf{i}}.$$

 The \textit{Weil-Petersson metric} $||\cdot||_{WP}$ is the Hermitian
metric on $T_{g}$ arising from the the \textsl{Petersson scalar  product}
\begin{equation}
 <\varphi,\psi>= \int_X \frac{\varphi(z) \cdot \overline{\psi(z)}}{\sigma(z)}\frac{dz\wedge d\overline{z}}{-2\textbf{i}} \nonumber
\end{equation}
via duality.

Both the Teichm\"uller metric and the Weil-Petersson metric are $\Mod(S_g)$-invariant.\\

Let $||\cdot||_1$ and $||\cdot||_2$ be any two metrics on $T_g$. We call $||\cdot||_1$ is \textsl{controlled above} by $||\cdot||_2$ if there exists a positive constant $k$ such that 
$$||\cdot||_1 \leq k ||\cdot||_2$$
which is denoted by $||\cdot||_1 \prec ||\cdot||_2$. 

The Cauchy-Schwarz inequality and the Gauss-Bonnet formula tell us that 

$$||\cdot||_{WP} \prec ||\cdot||_T.$$

We call the two metrics $||\cdot||_1$ and $||\cdot||_2$ are $\textit{equivalent}$ (or $\textit{quasi-isometric}$) if 
$$||\cdot||_1 \prec ||\cdot||_2 \quad \textit{and} \quad ||\cdot||_2 \prec ||\cdot||_1.$$  We denote it by $||\cdot||_1 \approx ||\cdot||_2$.

It is not hard to see that $||\cdot||_{WP}$ is not equivalent to $||\cdot||_{T}$ because the Weil-Petersson metric is incomplete and the Teichm\"uller metric is complete.  

\subsection{K\"ahler metrics on $\mathbb{M}_g$} In this subsection we briefly review some properties of the following three K\"ahler metrics $\mathbb{M}_g$: the McMullen metric, the Ricci metric, and the perturbed Ricci metric. They will be applied to prove Theorem \ref{mt-2} in section \ref{sec-4}. 
\subsubsection{McMullen metric} In \cite{McM00} McMullen constructed a new metric $||\cdot||_{M}$ on $\mathbb{M}_g$, called the \textsl{McMullen metric}. More precisely, let $Log: \mathbb{R}_{+}\rightarrow [0,\infty)$ be a smooth function such that\\ 
$(1). \ Log(x)=\log(x) \ \  if \ \ x \geq 2;$\\
$(2). \ Log(x)=0 \ \  if \ \ x \leq 1.$\\

For suitable choices of small constants $\epsilon, \delta>0$, the K\"ahler form of the McMullen metric is
$$\omega_M=\omega_{WP}-\textbf{i}\delta \sum_{\ell_{\gamma}(X)<\epsilon}\partial \overline{\partial} Log \frac{\epsilon}{\ell_{\gamma}}$$
where the sum is taken over primitive short geodesics $\gamma$ on $X$. Restricted on certain thick part of $\mathbb{M}_g$ the McMullen metric is exactly the Weil-Petersson metric. McMullen in \cite{McM00} proved that this metric is K\"ahler hyperbolic and has bounded geometry. He also showed that  

\begin{theorem}[McMullen]\label{McM-00}
On the moduli space $\mathbb{M}_g$, $||\cdot||_{M} \approx ||\cdot||_T.$
\end{theorem}

\subsubsection{Ricci metric and perturbed Ricci metric}

In \cite {Tromba, Wol86} it is shown that the Weil-Petersson metric has negative sectional curvature. The negative Ricci curvature tensor defines a new metric $||\cdot||_{\tau}$ which is called the $\textit{Ricci metric}$ on $\mathbb{M}_g$. Trapani in \cite{Trapani92} proved $||\cdot||_{\tau}$ is a complete K\"ahler metric.

In \cite{LSY04} Liu-Sun-Yau perturbed the Ricci metric along the Weil-Petersson direction to give new metrics on $\mathbb{M}_g$ which are called the \textsl{perturbed Ricci metrics} denoted by $||\cdot||_{LSY}$. More precisely, let $\omega_{\tau}$ be the K\"ahler form of the Ricci metric, for any constant $C>0$, the K\"ahler form of the perturbed Ricci metric is

$$\omega_{LSY}=\omega_{\tau}+C\cdot \omega_{WP}.$$

In \cite{LSY04} the authors showed that both $||\cdot||_{\tau}$ and $||\cdot||_{LSY}$ have bounded geometry. By using Yau's generalized Schwarz Lemma \cite{Yau-sl-78} they also showed that
\begin{theorem}[Liu-Sun-Yau]\label{lsy-04}
On the moduli space $\mathbb{M}_g$, 
$$||\cdot||_{LSY} \approx ||\cdot||_{\tau} \approx  ||\cdot||_M.$$
\end{theorem}

Furthermore, in one of their subsequent papers \cite{LSY05} they showed that

\begin{theorem}[Liu-Sun-Yau]\label{lsy-05}
With a suitable choice of the constant $C$, there exists two positive numbers $C_1, C_2$ such that the Ricci curvature of the perturbed Ricci metric $||\cdot||_{LSY}$ satisfies
$$-C_1\leq Ric_{||\cdot||_{LSY}}\leq -C_2<0.$$ 
\end{theorem}

Moreover, they also showed in \cite{LSY05} that the perturbed Ricci metric $||\cdot||_{LSY}$ has negatively pinched holomorphic sectional curvatures, which is known to be the first complete metric on the moduli space with this property. And they use this property and the Schwarz-Yau lemma to prove a list of canonical metrics on the moduli space are equivalent.

\subsection{The ratio of intermediate volume elements}\label{2.3}
In this subsection we briefly review a formula of Chern in \cite{Chern-68-ratio} which is crucial for this article. For the notations and computations one can also refer to \cite{Lu-68} for more details.

Let $M$ and $N$ be two $2n$-dimensional Hermitian manifolds and $f:M\rightarrow N$ be a holomorphic mapping. Let $\{\theta_i\}_{1\leq i \leq n}, \{\omega_{\alpha}\}_{1\leq \alpha \leq n}$ be the unitary coframe fields of $M$ and $N$ respectively. There exists complex numbers $a_{\alpha i}$ such that 
\begin{equation}
f^*\omega_\alpha=\sum_{i=1}^{n} a_{\alpha i} \theta_i.
\end{equation}

Then

\begin{equation}\label{2-1}
f^*(\omega_\alpha \wedge \overline{\omega}_\alpha )=\sum_{i=1}^{n} \sum_{j=1}^{n}a_{\alpha i}\overline{a}_{\alpha j} \theta_i \wedge \overline{\theta}_j.
\end{equation}

By raising equation (\ref{2-1}) to the $n^{th}$ power, the \textsl{ratio of intermediate volume elements} $v$ is defined as  

\begin{equation}\label{2-2}
v:=\frac{f^*(\sum_{\alpha=1}^{n}\omega_\alpha \wedge \overline{\omega}_\alpha )^n}{(\sum_{i=1}^{n}\theta_i \wedge \overline{\theta}_i )^n}.
\end{equation}

Linear algebra gives that 

\begin{equation}\label{2-3}
v=\frac{f^{*}(\dVol_N)}{\dVol_M}=D \cdot \overline{D} 
\end{equation}
where
\begin{equation}\label{2-4}
D=\det(a_{\alpha \beta}).
\end{equation}

Let $\Delta$ be the Laplace operator of $M$. Now we are ready to state Chern's formula.

\begin{theorem}[Chern]\label{chern-1} 

\[\frac{\Delta v}{4}=\sum_{k=1}^{n}D_k \cdot \overline{D_k}+\frac{v}{2}(\Sca-\sum_{1\leq \alpha, \beta, k \leq n } a_{\alpha k} \overline{a}_{\beta k} \tilde{Ric}_{a\overline{\beta}})\]
\end{theorem} 
where $\sum_{k=1}^{n}D_k \cdot \overline{D_k}$ is a nonnegative function on $M$, $\Sca$ is the scalar curvature of $M$ and $ \tilde{Ric}_{a\overline{\beta}}$ is the Ricci tensor of $N$. For the the proof of Theorem \ref{chern-1} one can refer to \cite{Chern-68-ratio} (or Corollary 4.4 in \cite{Lu-68}) for details. \\

\section{Proofs of Theorem \ref{mt-1}}\label{sec-3}
In this section we will prove Theorem \ref{mt-1}.

Let $M$ be any finite cover of the moduli space $\mathbb{M}_g$. If necessary we take a finite cover of $M$ again, still denoted by $M$, such that $M$ is a manifold. We lift the perturbed Ricci metric $||\cdot||_{LSY}$  in Theorem \ref{lsy-05} onto $M$. Let $||\cdot||$ be a complete finite-volume Hermitian metric on $M$. We consider the identity map

$$i:(M, ||\cdot||) \rightarrow (M, ||\cdot||_{LSY}).$$  

It is clear that $i$ is holomorphic. We let $v$ be the ratio of intermediate volume elements for $i$ as in equation (\ref{2-2}).

\begin{lemma}\label{3-1}
For any $p \in (M, ||\cdot||)$,

$$v(p)>0.$$

\end{lemma} 

\begin{proof}
For any $p \in (M, ||\cdot||)$, $v(p)>0$ follows from the fact that $i:(M, ||\cdot||) \rightarrow (M, ||\cdot||_{LSY})$ is biholomorphic.
\end{proof}

Now we are ready to prove Theorem \ref{mt-1}. 

\begin{proof}[Proof of Theorem \ref{mt-1}]
Assume that $\inf_{p \in (M, ||\cdot||)} \Sca(p) \geq 0$. We will argue it by getting a contradiction. Let $\Delta$ be the Laplace operator on $(M, ||\cdot||)$.\\

\textsl{Claim 1: $\Delta v(q) >0$ for any $q \in (M, ||\cdot||)$.}  \\

\begin{proof}[Proof of Claim 1] First by Theorem \ref{lsy-05} we know that there exists a constant $C_2>0$ such that 
\begin{equation}\label{3-1-1}
Ric_{||\cdot||_{LSY}}\leq -C_2<0.
\end{equation}

From Lemma \ref{3-1} we know that $v>0$. Since $\inf_{p \in (M, ||\cdot||)} \Sca(p) \geq 0$, by Theorem \ref{chern-1} we have
\begin{eqnarray}\label{3-1-2}
\Delta v &\geq& -2v (\sum_{1\leq \alpha, \beta, k \leq n } a_{\alpha k} \overline{a}_{\beta k} \tilde{Ric}_{a\overline{\beta}}).
\end{eqnarray}

Since $v>0$, inequalities (\ref{3-1-1}) and (\ref{3-1-2}) lead to

\begin{eqnarray}
\Delta v &\geq& 2 C_2 v \cdot (\sum_{1\leq \alpha,k \leq n } |a_{\alpha k}|^2)\\
&\geq&  2 C_2 (3g-3) v^{\frac{3g-2}{3g-3}} \quad \text{(by the AM-GM inequality)}\\
&>&0.
\end{eqnarray}
\end{proof}

The remaining argument is inspired by the proof of Theorem 9.1 in \cite{Canary93} and section 8.12 in \cite{Thursbook}. We remark that it is not known that $v$ is bounded from above, so the following Claim (2) can not directly follow from the results in \cite{Yau76}.\\

\textsl{Claim 2: $v$ is a constant on $(M,||\cdot||)$.} 

\begin{proof}[Proof of Claim 2] 
Let $g_t$ denote the flow generated by the vector field $\grad v$. Since $(M,||\cdot||)$ is complete, $g_t$ is defined for all $t\geq 0$.  

Assume that $v$ is not a constant and let $p_0 \in M$ such that $\grad v(p_0)\neq 0$. Along the flow line of $g_t$ starting at $p_0$, $v$ is increasing since for all $s_2>s_1\geq 0$,
\begin{eqnarray}\label{3-2-1}
v(g_{s_2}(p_0))-v(g_{s_1}(p_0))=\int_{s_1}^{s_2}||\grad v (g_t(p_0))||dt \geq 0.
\end{eqnarray} 
That is 
\begin{eqnarray}
v(g_{s_2}(p_0))\geq v(g_{s_1}(p_0)) \quad \forall s_2>s_1\geq 0.
\end{eqnarray} 

Since we assume that $\grad v(p_0)\neq 0$, let $s_2=1$ and $s_1=0$ we have
\begin{eqnarray}
v(g_{1}(p_0))> v(p_0)>0.
\end{eqnarray} 

Therefore there exists a small enough constant $r_0>0$ such that 
\begin{eqnarray}
\inf_{q \in B(p_0,r_0)} v(g_{1}(q))> \sup_{q \in B(p_0,r_0)}v(q)
\end{eqnarray} 
where $B(p_0,r_0)$ is the geodesic ball centered at $p_0$ of radius $r_0$. 

In particular we have 
\begin{eqnarray}\label{3-2-2}
B(p_0,r_0) \cap g_{1}(B(p_0,r_0))=\emptyset.
\end{eqnarray}  

Inequality (\ref{3-2-1}) and equation (\ref{3-2-2}) give that
\begin{eqnarray}\label{3-2-3}
B(p_0,r_0) \cap g_{n}(B(p_0,r_0))=\emptyset \quad \forall n\in \mathbb{Z}^{+}.
\end{eqnarray}  

Which also implies
\begin{equation}\label{3-3-3}
g_{n }(B(p_0,r_0)) \cap g_{m}(B(p_0,r_0))=\emptyset \quad \forall n \neq m \in \mathbb{Z}^+.
\end{equation}
Otherwise there exist two positive integers $n_0 > m_0\geq 1$ and $q_1, q_2 \in B(p_0,r_0)$ such that $g_{n_0}(q_1)=g_{m_0}(q_2)$. Since $g_t$ is a flow, $g_{n_0-m_0}(q_1)=q_2$ which contradicts equation (\ref{3-2-3}).\\

On the other hand, for any $t_0>0$ (we use Proposition 18.18 in \cite{Lee-manifold}), we have 
\begin{eqnarray}
\frac{d \Vol(g_t(B(p_0,r_0)))}{dt}|_{t=t_0}&=& \int_{B(p_0,r_0)} \frac{d}{dt}|_{t=t_0} g_t^*(\dVol) \\
&=& \int_{B(p_0,r_0)} g_{t_0}^*(\L_{\grad v}(\dVol)) \\
&=&\int_{B(p_0,r_0)} g_{t_0}^*(\Div( \grad (v))\dVol)\\
&=& \int_{g_{t_0}(B(p_0,r_0))} \Delta v  \dVol.
\end{eqnarray}
From Claim (1) we have
\begin{eqnarray}\label{3-2-4}
\frac{d \Vol(g_t(B(p_0,r_0)))}{dt}|_{t=t_0}>0 \quad \forall t_0>0.
\end{eqnarray}
That is the flow $g_t$ is volume increasing. 

Thus, equation (\ref{3-3-3}) and inequality (\ref{3-2-4}) give that 

\begin{eqnarray}
\Vol((M,||\cdot||))&\geq& \Vol(\cup_{k=1}^{\infty}g_{k }(B(p_0,r_0)))\\
&=&\sum_{k=1}^{\infty} \Vol (g_{k }(B(p_0,r_0)))\\
&\geq &  \sum_{k=1}^{\infty} \Vol (B(p_0,r_0))\\
&=& \infty
\end{eqnarray}
which contradicts our assumption that $(M,||\cdot||)$ has finite volume.
\end{proof}

It is clear that Claim (1) and Claim (2) can not simultaneously hold. Therefore, the proof is completed.

\end{proof}

\section{Proofs of Theorem \ref{mt-2}}\label{sec-4}

As the same in section \ref{sec-3} we let $M$ be any finite cover of the moduli space $\mathbb{M}_g$ and $||\cdot||_{LSY}$ be the perturbed Ricci metric in Theorem \ref{lsy-05}.  
Let $k_1, k_2$ be two positive constants and $||\cdot||$ be a Hermitian metric on $M$ with
 $$ k_1 ||\cdot|| \leq ||\cdot||_T \leq k_2 ||\cdot|| .$$

From Theorem \ref{McM-00} and Theorem \ref{lsy-04}, up to some uniform constants, we may assume that
\begin{equation}\label{eq-4-1}
k_1 ||\cdot|| \leq ||\cdot||_{LSY} \leq k_2 ||\cdot|| .
\end{equation}

Use the same notations in section \ref{sec-3} we let $v$ be the ratio of intermediate volume elements for $i$ as in equation (\ref{2-2}).

\begin{lemma}\label{4-1}
For any $p \in (M, ||\cdot||)$, we have
$$v(p)\geq k_1^{6g-6}>0.$$
In particular, $(M, ||\cdot||)$  has finite volume.

\end{lemma} 

\begin{proof}
Since $k_1||\cdot||\leq   ||\cdot||_{LSY}$, linear algebra gives that
\begin{equation}\label{4-1-1-2}
 k_1^{6g-6} \dVol_{||\cdot||} \leq \dVol_{ ||\cdot||_{LSY}}.
\end{equation}
The conclusion follows from equation (\ref{2-3}) and inequality (\ref{4-1-1-2}).

Since $(M, ||\cdot||_{LSY})$  has finite volume, inequality (\ref{4-1-1-2}) tells that $(M, ||\cdot||)$ also has finite volume.

\end{proof}

\begin{lemma}\label{4-2}
$(M, ||\cdot||)$ is complete.
\end{lemma} 
\begin{proof}
It directly follows from our assumption that $||\cdot||_T \leq k_2 ||\cdot||$ and the fact that the Teichm\"uller metric is complete. 
\end{proof}

Now we are ready to prove Theorem \ref{mt-2}. 

\begin{proof}[Proof of Theorem \ref{mt-2}]
Theorem \ref{lsy-05} tells that there exists a constant $C_2>0$ such that 
\begin{equation}\label{ubfr}
Ric_{||\cdot||_{LSY}}\leq -C_2<0.
\end{equation}

We let $k_3:= \inf_{p \in (M, ||\cdot||)} \Sca(p) $.
From Lemma \ref{3-1} we know that $v>0$. Thus, from Theorem \ref{chern-1} we have
\begin{eqnarray*}\label{4-1-1}
\Delta v &\geq& 2 v \cdot k_3    -2v (\sum_{1\leq \alpha, \beta, k \leq n } a_{\alpha k} \overline{a}_{\beta k} \tilde{Ric}_{a\overline{\beta}})\\
&\geq&  2 v( k_3+C_2 (3g-3) v^{\frac{1}{3g-3}}) \quad \text{(by the AM-GM inequality)}.
\end{eqnarray*}
From Lemma \ref{4-1} we have

\begin{eqnarray}\label{4-1-2}
\Delta v &\geq& 2 k_1^{6g-6} \cdot (k_3 +C_2 (3g-3) k_1^2).
\end{eqnarray}

We choose $K(k_1,g)= \frac{C_2 (3g-3) k_1^2}{2}>0$. \\

\textsl{Claim : $\inf_{p \in (M, ||\cdot||)} \Sca(p) \leq -K(k_1,g)<0$.}  

\begin{proof}[Proof of Claim] Assume it is not. That is 
\begin{eqnarray}\label{4-1-3}
\inf_{p \in (M, ||\cdot||)} \Sca(p) \geq -K(k_1,g).
\end{eqnarray}

Inequalities (\ref{4-1-2}) and (\ref{4-1-3}) lead to

\begin{eqnarray}\label{4-1-4}
\Delta v \geq  C_2 \cdot (3g-3)\cdot k_1^{6g-4} >0.
\end{eqnarray}

There are more information coming from the conditions of Theorem \ref{mt-2}. We provide two different methods to finish the proof of the claim.

\begin{proof}[Method (1).]
Since $(M,||\cdot||)$ is complete and has finite volume, it follows from inequality (\ref{4-1-4}) and the same argument as in the \textsl{proof of Claim (2)} in section \ref{sec-3} that $v$ is a constant which contradicts inequality (\ref{4-1-4}). 
\end{proof}

\begin{proof}[Method (2)]
First from the right side of inequality (\ref{eq-4-1}) and by using a same argument in the proof of Lemma \ref{4-1} there exists a positive constant $C_3$ such that 
\begin{eqnarray}\label{4-1-5}
\sup_{p\in (M,||\cdot||)}v(p)\leq C_3.
\end{eqnarray}

Fix $p_0 \in M$ and let $B(p_0,r)$ be the closed geodesic ball of $(M, ||\cdot||)$ centered at $p_0$ of radius $r$. Then for any $t>0$ there exists a bump function $f(x):(M, ||\cdot||) \rightarrow [0,\infty)$ which is a Lipschitz continuous function and a constant $C_4>0$ such that 

(i). $f\equiv 1$ on $B(p_0, t)$ and $f \equiv 0$ on $M-B(p_0,2t)$.

(ii). $||\grad f||\leq \frac{C_4}{t}$ a.e. on $M$.  

(For the existence of such bump functions one can refer to \cite{Yau76}).

First since $v>0$ and $\Delta v>0$ we have 

\begin{equation}\label{3-2-2}
\Delta v^2 \geq 2 ||\grad v||^2.
\end{equation}

Let $<,>$ be the Riemannian inner product associated to the metric $||\cdot||$. The Stokes' theorem and inequality (\ref{3-2-2}) give 

\begin{eqnarray*}
0&=&\int_{B(p_0,2t)}div (f^2 \grad (v^2))\\
&\geq& 4 \int_{B(p_0,2t)}f \cdot v \cdot <\grad f, \grad v>+ 2\int_{B(p_0,2t)}f^2 \cdot ||\grad v||^2.
\end{eqnarray*}

The Cauchy-Schwarz inequality leads to

\begin{eqnarray*}
\int_{B(p_0,2t)}f^2 \cdot  ||\grad v||^2 &\leq& -2 \int_{B(p_0,2t)}f \cdot v \cdot <\grad f, \grad v>\\ 
&\leq& 2\sqrt {\int_{B(p_0,2t)}f^2 \cdot ||\grad v||^2} \cdot \sqrt {\int_{B(p_0,2t)}v^2 \cdot ||\grad f||^2}.
\end{eqnarray*}

That is
\begin{eqnarray}\label{3-3}
\int_{B(p_0,2t)}f^2 \cdot  ||\grad v||^2 
\leq 4\int_{B(p_0,2t)}v^2\cdot ||\grad f||^2.
\end{eqnarray}

Since $f\equiv 1$ on $B(p_0, t)$ and $||\grad f||\leq \frac{C_4}{t}$ a.e. on $M$, we replace inequality (\ref{3-3}) by

\begin{eqnarray}
\int_{B(p_0,t)} ||\grad v||^2 
\leq \frac{4 C_4^2}{t^2} \int_{B(p_0,2t)}v^2 .
\end{eqnarray}

By inequality (\ref{4-1-5}),

\begin{eqnarray}
\int_{B(p_0,t)} ||\grad v||^2 
\leq \frac{4 C_3^2 C_4^2}{t^2} \Vol(B(p_0, 2t)).
\end{eqnarray}

Since we assume that $(M,||\cdot||)$ has finite volume, there exists a constant $C_5>0$ such that 

\begin{eqnarray}\label{3-2}
\int_{B(p_0,t)} ||\grad v||^2
\leq \frac{C_5}{t^2}.
\end{eqnarray}

Since $(M,||\cdot||)$ is complete and open, we let $t\to \infty$, it follows from inequality (\ref{3-2}) that $\grad v\equiv 0$ on $M$. That is, $v$ is a constant on $M$ which contradicts inequality (\ref{4-1-4}).
\end{proof}
For the second method above, we use the same argument as in \cite{Yau76}.
\end{proof}

It is clear that the conclusion follows from the claim.
\end{proof}

\section{Proof of Theorem \ref{mt-3}}\label{sec-5}
Recall a Hermitian manifold $(M,<,>)$ is \textsl{almost} if there exists an almost complete structure $\mathbb{J}$ on $M$ such that 
$$<\mathbb{J}\circ V, \mathbb{J} \circ W>=<V,W>$$
for all tangent vectors $V$ and $W$. It is clear that a K\"ahler manifold is almost Hermitian. 

If we assume that both $M$ and $N$ in subsection \ref{2.3} are almost Hermitian, Goldberg and Har\'el in \cite{GH79} proved that the term $\sum_{k=1}^{n}D_k \cdot \overline{D_k}$ in Chern's formula (see Theorem \ref{chern-1}) satisfies 
$$\sum_{k=1}^{n}D_k \cdot \overline{D_k}=\frac{||\grad v||^2}{4v}.$$ 

Thus, a direct computation for Theorem \ref{chern-1}  gives that
\begin{equation}\label{chern-2}
\Delta \log(v)=2\cdot(\Sca-\sum_{1\leq \alpha, \beta, k \leq n } a_{\alpha k} \overline{a}_{\beta k} \tilde{Ric}_{a\overline{\beta}}).
\end{equation}
(One can see formula (10) in \cite{GH79} for details.)

Before we prove Theorem \ref{mt-3} let us recall a theorem of S. T. Yau (see Theorem 1 in \cite{Yau76}) which is crucial for this section.

Let $N$ be a complete Riemannian manifold and $\Delta$ be the Laplace operator of $N$. Assume that $u$ and $h$ are two functions on $M$ satisfying the following equation
\begin{equation}\label{4-3-1}
\Delta \log u=h.
\end{equation}
Then the following theorem says that

\begin{theorem}[Yau]\label{yau-2}
 Suppose $h$ is bounded from below by a constant and $0<\int_{N}h(p)\dVol(p)\leq \infty$. Then $\int_{N}u^n(p)\dVol(p)=\infty$ for $n>0$, unless $u$ is a constant function.
\end{theorem}

Let $M$ be a finite cover of the moduli space and $|\cdot||$ be an almost Hermitian metric with $||\cdot||\approx ||\cdot||_T$. We use the same notations here as in section \ref{sec-3}. In our setting we let 
\begin{equation*}
N=(M,||\cdot||), \quad u=v
\end{equation*}

and

$$h=2\cdot(\Sca-\sum_{1\leq \alpha, \beta, k \leq (3g-3) } a_{\alpha k} \overline{a}_{\beta k} \tilde{Ric}_{a\overline{\beta}}).$$

Since the perturbed Ricci metric $||\cdot||_{LSY}$ is K\"ahler and $(M,||\cdot||)$ is almost Hermitian, formula (\ref{chern-2}) exactly tells us that

\begin{equation}\label{4-3-2}
\Delta \log v=h.
\end{equation}

Now we are ready to prove Theorem \ref{mt-3}.

\begin{proof}[Proof of Theorem \ref{mt-3}]
First since $||\cdot||\approx ||\cdot||_T$, From Theorem \ref{McM-00} and Theorem \ref{lsy-04}, we may assume that
\begin{equation}\label{eq-4-3-1}
k_1 ||\cdot|| \leq ||\cdot||_{LSY} \leq k_2 ||\cdot|| .
\end{equation}

The proof of Lemma \ref{4-1} gives that 

\begin{equation}\label{eq-4-3-2}
k_1^{6g-6}\leq v \leq k_2^{6g-6}.
\end{equation}

By using a similar argument in the previous proofs, from Theorem \ref{lsy-05} we have, for any $p\in (M,||\cdot||)$,

\begin{eqnarray}\label{4-3-3}
h(p)&=&2\cdot(\Sca(p)-\sum_{1\leq \alpha, \beta, k \leq (3g-3) } a_{\alpha k}(p) \overline{a}_{\beta k}(p) \tilde{Ric}_{a\overline{\beta}}(p)) \nonumber \\
&\geq& 2 \cdot \inf_{p \in (M,||\cdot||))}\Sca(p)  + 2C_2 (\sum_{1\leq \alpha, \beta\leq (3g-3) } |a_{\alpha \beta}(p)|^2)  \nonumber \\
&\geq&  2 \cdot \inf_{p \in (M,||\cdot||))}\Sca(p) +2C_2 (3g-3) v^{\frac{1}{3g-3}} \quad \text{(by the AM-GM inequality)} \nonumber \\
&>& 2\cdot \inf_{p \in (M,||\cdot||))}\Sca(p) \nonumber\\
&>&-\infty\nonumber 
\end{eqnarray}
where we apply the assumption that the scalar curvature of $(M,||\cdot||)$ is bounded from below for the last step. That is, $h$ is bounded from below on  $(M,||\cdot||)$.

We finish the proof through the following two cases.\\

\textsl{Case (1). $v$ is a constant on $(M,||\cdot||)$.}

If $v$ is a constant, by equation (\ref{4-3-2}) we have $h \equiv 0$. That is,

\begin{equation}\label{4-3-5}
\Sca=\sum_{1\leq \alpha, \beta, k \leq n } a_{\alpha k} \overline{a}_{\beta k} \tilde{Ric}_{a\overline{\beta}}.
\end{equation}

It follows from Theorem \ref{lsy-05} and equation (\ref{4-3-5}) that 

\begin{equation}\label{4-3-6}
\Sca(p)<0 \quad \forall p \in (M,||\cdot||).
\end{equation}

It is clear that the total mass

$$\int_{p \in (M,||\cdot||)} \Sca(p)\dVol(p)<0.$$\\

\textsl{Case (2). $v$ is not a constant on $(M,||\cdot||)$.}

Since  $(M,||\cdot||_{LSY})$ has finite volume, from inequality (\ref{eq-4-3-2}) we have

\begin{equation}\label{eq-4-3-4}
 \int_{p \in (M,||\cdot||)}v^n(p)\dVol(p)<\infty \quad \forall n>0.
\end{equation}

We already show that $h$ is bounded from below on $(M,||\cdot||)$. Thus, from Theorem \ref{yau-2} we know that

\begin{equation}\label{4-3-7}
\int_{p \in (M,||\cdot||)} h(p)\dVol(p)\leq 0.
\end{equation}

That is,

\begin{eqnarray*}
\int_{p \in (M,||\cdot||)} \Sca(p)\dVol(p)\leq \int_{p \in (M,||\cdot||)}\sum_{1\leq \alpha, \beta, k \leq (3g-3)}a_{\alpha k}(p) \overline{a}_{\beta k}(p) \tilde{Ric}_{a\overline{\beta}}(p)\dVol(p).
\end{eqnarray*}

By Theorem \ref{lsy-05} and the AM-GM inequality we have

\begin{eqnarray*}
\int_{p \in (M,||\cdot||)} \Sca(p)\dVol(p) &\leq& -\int_{p \in (M,||\cdot||)} C_2 (\sum_{1\leq \alpha, \beta\leq (3g-3) } |a_{\alpha \beta}(p)|^2)\dVol(p)\\
&\leq & -\int_{p \in (M,||\cdot||)}C_2 (3g-3) v(p)^{\frac{1}{3g-3}} \dVol(p)\\
&<&0.
\end{eqnarray*}
Where we apply the fact that $v>0$ for the last step.

\end{proof}

%\begin{remark}
%Based on the results in this paper it is interesting to know \textsl{whether there exists a Hermitian metric $||\cdot||$ on a finite %cover $M$ of $\mathbb{M}_g$ such that $||\cdot||$ is equivalent to the Teichm\"uller metric and $(M,||\cdot||)$ has non-negative %scalar curvature outside some compact subset of $M$.} We hope the method in this article is helpful for this question. If one %carefully check the proof of Theorem \ref{mt-1}, one possible way is to study \textsl{whether the equivalence of the Teichm\"uller %metric could help to find a flow line $\{g_{t}(p_0)\}_{t\geq 0}$ for $\grad v$, the gradient of the ratio of intermediate volume %elements defined in section \ref{sec-3}, such that the line $\{g_{t}(p_0)\}_{t\geq 0}$ leaves every thick-part of the moduli space.}
%\end{remark}

\section{acknowledgement}
The author would like to thank Kefeng Liu, Michael Wolf and Scott Wolpert for their interests. He deeply thank Benson Farb for his encouragement for writing this article. This work was partially completed during a reading seminar at Rice university, organized by Michael Wolf and Robert Hardt. The author would like to thanks for their consistent help. He also would like to acknowledge support from U.S. National Science Foundation grants DMS 1107452, 1107263, 1107367 ``RNMS: Geometric structures And Representation varieties"(the GEAR Network).

\end{document}